\documentclass[11pt,a4paper]{amsart}

\usepackage{mathtools}
\usepackage{amsthm}
\usepackage{amssymb}
\usepackage{hyperref}
\usepackage{caption}

 \usepackage{tikz}

\theoremstyle{definition}
\newtheorem{definition}{Definition}[section]

\theoremstyle{plain}
\newtheorem{corollary}[definition]{Corollary}
\newtheorem{lemma}[definition]{Lemma}
\newtheorem{proposition}[definition]{Proposition}
\newtheorem{theorem}[definition]{Theorem}

\numberwithin{equation}{section}

\newcommand*\commgraph[1]{\mathcal{G}(#1)}
\newcommand*\extendedcommgraph[1]{\mathcal{G}^*(#1)}
\newcommand*\Rees[4]{\mathcal{M}[#1; \allowbreak #2,\allowbreak #3; \allowbreak #4]}

\newcommand*\diam[1]{\mathrm{diam}(#1)}
\newcommand*\girth[1]{\mathrm{girth}(#1)}
\newcommand*\knitdegree[1]{\mathrm{kd}(#1)}
\newcommand*\cliquenumber[1]{\omega(#1)}
\newcommand*\chromaticnumber[1]{\chi(#1)}
\newcommand*\graphjoin{\mathbin{\nabla}}

\newcommand*\defterm{\emph}

\DeclarePairedDelimiter{\abs}{\lvert}{\rvert}

\DeclarePairedDelimiter{\parens}{\lparen}{\rparen}
\DeclarePairedDelimiter{\bracks}{\lbrack}{\rbrack}
\DeclarePairedDelimiter{\braces}{\{}{\}}

\DeclarePairedDelimiter{\set}{\{}{\}}
\DeclarePairedDelimiterX{\gset}[2]{\{}{\}}{\,#1:#2\,}

\tikzset{
  x=10mm,
  y=10mm,
  treenode/.style={
    circle,
    minimum size=1.3mm,
    inner sep=0,
    fill=black,
  },
  edgelabel/.style={
    node font=\small,
    inner sep =.5mm,
  },
  columnlabel/.style={
    anchor=north,
  }
}

\tikzset{
  vertex/.style={
    circle,
    minimum size=3mm,
    fill,
    inner sep=0,
    outer sep=0,
  },
  edge/.style={
    line width=.5mm,
  }
}

\allowdisplaybreaks

\begin{document}

\title{Commuting graphs of completely simple semigroups}

\author{Tânia Paulista}
\address[T. Paulista]{%
Center for Mathematics and Applications (NOVA Math) \& Department of Mathematics\\
NOVA School of Science and Technology\\
NOVA University of Lisbon\\
2829--516 Caparica\\
Portugal
}
\email{%
t.paulista@campus.fct.unl.pt
}
\thanks{This work is funded by national funds through the FCT -- Fundação para a Ciência e a Tecnologia, I.P., under the scope of the projects UIDB/00297/2020 and UIDP/00297/2020 (Center for Mathematics and Applications). The author is also funded by national funds through the FCT -- Fundação para a Ciência e a Tecnologia, I.P., under the scope of the studentship 2021.07002.BD}

\thanks{The author is also thankful to her supervisors António Malheiro and Alan J. Cain for all the support, encouragement and guidance}


\begin{abstract}
We describe the commuting graph of a Rees matrix semigroup over a group and investigate its properties: diameter, clique number, girth, chromatic number and knit degree. The maximum size of a commutative subsemigroup of a Rees matrix semigroup over a group is presented, and its largest commutative subsemigroups are exhibited. We use the knowledge we obtained from the commuting graph of this semigroup construction to deduce results regarding the properties of commuting graphs of completely simple semigroups. We also characterize the graphs that arise as commuting graphs of completely simple semigroups. In the process of obtaining these results we are also able to restrict the possible values for some properties of commuting graphs of groups.
\end{abstract}

\maketitle

\section{Introduction}

The notion of commuting graph was initially introduced exclusively for groups. The commuting graphs of several groups were studied in terms of their properties; for example, the symmetric and the alternating groups \cite{Commuting_graph_I(X), Symmetric_group, Diameter_commuting_graph_symmetric_group, Commuting_graph_symmetric_alternating_groups, Alternating_group}. The study of the combinatorial structure of the commuting graph of a group can provide information regarding the group --- for instance, the clique number of the commuting graph of a group is equal to the difference between the maximum size of an abelian subgroup of the given group and the number of central elements of that group. For this reason, commuting graphs have been useful in proving group theoretical questions (see, for example, \cite{Importance_commuting_graphs_1, Importance_commuting_graphs_2, Importance_commuting_graphs_3}).

Given how helpful commuting graphs turned out to be for groups, the notion ended up being extended to semigroups. Just as happens with groups, the commuting graph of a semigroup also allows a better understanding of the algebraic structure of a semigroup.  Some authors have already determined some properties of the commuting graphs of some important semigroups, such as the semigroup of full transformations \cite{Commuting_graph_T(X)}, and the symmetric inverse semigroup \cite{Commuting_graph_I(X)}. Commuting graphs of semigroups have also been used to solve some semigroup problems. For instance, Araújo, Kinyon and Konieczny \cite{Commuting_graph_T(X)} introduced the notion of left paths and knit degree (in commuting graphs) in order to give an answer to a conjecture that Schein established in the process of determining a characterization for $r$-semisimple bands \cite{Schein_conjecture}.

In this paper we consider the class of completely simple semigroups and investigate its commuting graphs. It is well known that the class of completely simple semigroups includes groups, and that these semigroups are generally close to groups. In fact, completely simple semigroups share some properties with groups: all the elements of a completely simple semigroup lie in some subgroup of the semigroup and its $\mathcal{H}$-classes are groups. Furthermore, completely simple semigroups are simple, which means that their unique ideal is the semigroup itself, and all their idempotents are minimal. In order to know more about the commuting graphs of (non-group) completely simple semigroups, we examine the commuting graph of a semigroup construction called the Rees matrix semigroup over a group, which is used to characterize completely simple semigroups. Our aim is to determine some of its properties (namely its diameter, clique number, girth, chromatic number and knit degree) and then use that information to deduce properties of the commuting graphs of completely simple semigroups. More specifically, we are interested in determining the values that can be achieved by those properties when we consider the commuting graph of a completely simple semigroup.

Some authors have shown interest in determining the possible values for some properties of the commuting graph of a semigroup/group. In 2011 Araújo, Kinyon and Konieczny \cite{Commuting_graph_T(X)} proved that for every positive integer $n\geqslant 2$ there exists a semigroup whose commuting graph has diameter $n$. Araújo, Kinyon and Konieczny also showed that for every positive integer $n\geqslant 2$ such that $n\neq 3$ there exists a semigroup with knit degree equal to $n$. The existence of a semigroup with knit degree $3$ was only demonstrated in 2016 by Bauer and Greenfeld \cite{Graphs_that_arise_as_commuting_graphs_of_semigroups}. They also provided a characterization of the graphs that can arise as commuting graphs of semigroups. Giudici and Kuzma \cite{Graphs_that_arise_as_commuting_graphs_of_semigroups_2} obtained independently the same description. This characterization was fundamental in confirming the existence of semigroups whose commuting graph has clique number/chromatic number equal to $n$, for each positive integer $n$. More recently, Cutolo \cite{Semigroup_whose_commuting_graph_has_diameter_n} proved that for each positive integer $n\geqslant 3$ there exists a group whose commuting graph has diameter $n$. This result improved the one presented by Guidici and Parker in 2013 \cite{Semigroup_whose_commuting_graph_has_diameter_greater_than_n}, which stated that for each positive integer $n$ there is a group whose commuting graph has diameter greater than $n$.

The remainder of this paper is organized in three sections. Section~\ref{Preliminaries} provides some background on graphs and their properties, commuting graphs, Rees matrix semigroups over groups and completely simple semigroups. Section~\ref{Commuting graph of a Rees matrix semigroup over a group} includes results regarding the commuting graph of a Rees matrix semigroup over a group. Here we characterize these graphs, determine their diameter, clique number, girth, chromatic number, and examine the existence of left paths. We also find the maximum size of a commutative subsemigroup of a Rees matrix semigroup over a group, as well as the commutative subsemigroups that achieve that maximum size. Finally, in Section~\ref{Properties of commuting graphs of completely simple semigroups} we find the possible values for the clique number, girth and chromatic number of a commuting graph of a completely simple semigroup, and the possible values for the knit degree of a completely simple semigroup.

\section{Preliminaries} \label{Preliminaries}

\subsection{Graphs}\label{Subsection graphs}

Let $G=(V,E)$ be a simple graph, that is, an undirected graph without loops and multiple edges.

A \defterm{path} in $G$ from a vertex $x$ to a vertex $y$ is a sequence of pairwise distinct vertices (except, possibly, $x$ and $y$) $x=x_1,x_2,\ldots,x_n=y$ such that $\braces{x_1,x_2}, \braces{x_2,x_3},\ldots, \braces{x_{n-1},x_n}$ are pairwise distinct edges of $G$. The \defterm{length} of the path is $n-1$, which corresponds to the number of edges of the path. We say that $G$ is \defterm{connected} if for every vertices $x$ and $y$ there is a path from $x$ to $y$. The \defterm{distance} between the vertices $x$ and $y$, denoted $d(x,y)$, is the length of a shortest path from $x$ to $y$. If there is no such path between the vertices $x$ and $y$, then the distance between $x$ and $y$ is defined to be infinity, that is, $d(x,y)=\infty$. The \defterm{diameter} of $G$, denoted $\diam{G}$, is the maximum distance between vertices of $G$, that is, $\diam{G}=\max\left\{d(x,y):x,y\in V\right\}$. We notice that the diameter of $G$ is finite if and only if $G$ is connected.

Let $K\subseteq V$. We say that $K$ is a \defterm{clique} in $G$ if $\braces{x,y}\in E$ for all $x,y\in K$. The \defterm{clique number} of $G$, denoted $\cliquenumber{G}$, is the size of a largest clique in $G$, that is, $\cliquenumber{G}=\max\left\{|K|: K \text{ is a clique in } G\right\}$.

Assume that the graph $G$ contains cycles. The the \defterm{girth} of $G$, denoted $\girth{G}$, is the length of a shortest cycle in $G$.

The \defterm{chromatic number} of $G$, denoted $\chromaticnumber{G}$, is the minimum number of colours required to colour the vertices of $G$ in a way such that adjacent vertices have different colours.

Throughout this paper we will mention some graphs, namely complete graphs, join of graphs and induced graphs. Below we provide a remainder of the definition of each one of them. 

A \defterm{complete graph} is a graph where all distinct vertices are adjacent to each other. The unique (up to isomorphism) complete graph with $n$ vertices is denoted $K_n$.

Given two simple graphs $G=(V,E)$ and $H=(V',E')$, their \defterm{graph join}, denoted $G\graphjoin H$, is defined to be the new (simple) graph whose set of vertices is $V\cup V'$, and where two vertices $x,y\in V\cup V'$ are adjacent if and only if one of the following conditions is verified:
\begin{enumerate}
    \item $x\in V$ and $y\in V'$ (or vice versa).
    \item $x,y\in V$ and $\braces{x,y}\in E$ (or $x,y\in V'$ and $\braces{x,y}\in E'$).
\end{enumerate}
This means that $G\graphjoin H$ can be obtained from the graphs $G$ and $H$ by making all of the vertices of $G$ adjacent to all of the vertices of $H$.

Given a simple graph $G=(V,E)$ and $V'\subseteq V$, we call \defterm{subgraph induced by $V$} the subgraph of $G$ whose set of vertices is $V'$ and where two vertices are adjacent if and only if they are adjacent in $G$ (that is, the set of edges of the induced subgraph is $\braces{\braces{x,y}\in E: x,y\in V'}$).

\subsection{Commuting graphs and extended commuting graphs}\label{Subsection commuting graphs}

{\sloppy Ever since commuting graphs were introduced there have been variations of their definition. In this section we introduce two of those definitions. The main difference between them is the set of vertices of the graphs --- in one of them the vertices correspond to the elements of the semigroup, and in the other one we exclude the central elements of the semigroup from the set of vertices. \par}

The \defterm{center} of a semigroup $S$ is the set
\begin{displaymath}
    Z(G)= \left\{x\in S: xy=yx \text{ for all } y\in S\right\}.
\end{displaymath}

Let $S$ be a finite non-commutative semigroup. The \defterm{commuting graph} of $S$, denoted $\commgraph{S}$, is the simple graph whose set of vertices is $S\setminus Z(S)$ and where two distinct vertices $x,y\in S\setminus Z(S)$ are adjacent if and only if $xy=yx$.

Let $S$ be a finite semigroup. The \defterm{extended commuting graph} of $S$, denoted $\extendedcommgraph{S}$, is the simple graph whose set of vertices is $S$ and where two distinct vertices $x,y\in S$ are adjacent if and only if $xy=yx$.

Note that in the first definition the semigroup must be non-commutative (because otherwise we would obtain an empty graph), but in the second one we allow the semigroup to be commutative. Additionally, the second definition implies that the center of the semigroup is a clique in the extended commuting graph of the semigroup. 

The next lemma, which is easy to prove, gives a characterization of the extended commuting graph of a semigroup. When the semigroup is not commutative, this characterization shows a relation between the commuting graph and the extended commuting graph of the semigroup.

\begin{lemma}\label{commgraph and extended commgraph}
    Let $S$ be a finite semigroup.
    \begin{enumerate}
        \item If $S$ is commutative, then $\extendedcommgraph{S}$ is isomorphic to $K_{\abs{S}}$.

        \item If $S$ is non-commutative, then $\extendedcommgraph{S}$ is isomorphic to $K_{\abs{Z\parens{S}}}\graphjoin\commgraph{S}$.
    \end{enumerate}
\end{lemma}

The following definition corresponds to a concept created specifically for commuting graphs of semigroups.

Let $S$ be a non-commutative semigroup. A \defterm{left path} in $\commgraph{S}$ is a path $x_1,\ldots,x_n$ in $\commgraph{S}$ such that $x_1\neq x_n$ and $x_1x_i=x_nx_i$ for all $i\in\braces{1,\ldots,n}$. If $\commgraph{S}$ contains left paths then the \defterm{knit degree} of $S$, denoted $\knitdegree{S}$, is the length of a shortest left path in $\commgraph{S}$.
 
\subsection{Completely simple semigroups}\label{Subsection completely simple semigroups}
A semigroup without a zero is \defterm{completely simple} if it satisfies the following conditions:
\begin{enumerate}
    \item It contains no proper ideals.
    \item It contains a minimal idempotent.
\end{enumerate}

Completely simple semigroups can also be characterized via a semigroup construction called the Rees matrix semigroup over a group (Theorem~\ref{completely simple semigroup <=> Rees matrix construction}), which is described below.

Let $G$ be a group, let $I$ and $\Lambda$ be index sets, and let $P$ be a $\Lambda \times I$ matrix whose entries are elements of $G$. For each $i\in I$ and $\lambda \in \Lambda$, we denote by $p_{\lambda i}$ the $\parens{\lambda, i}$-th entry of the matrix $P$. A \defterm{Rees matrix semigroup over a group}, denoted $\Rees{G}{I}{\Lambda}{P}$, is the set $I\times G \times \Lambda$ with multiplication defined as follows
\begin{displaymath} \parens{i,x,\lambda}\parens{j,y,\mu}=\parens{i,xp_{\lambda j}y,\mu}.
\end{displaymath}

\begin{theorem}[{\cite[Theorem 4.11]{Nine_chapters_Cain}}]\label{completely simple semigroup <=> Rees matrix construction}
    A semigroup $S$ is completely simple if and only if there exist a group $G$, index sets $I$ and $\Lambda$, and a $\Lambda \times I$ matrix $P$ with entries from $G$ such that $S\simeq \Rees{G}{I}{\Lambda}{P}$.
\end{theorem}

This theorem is fundamental and we will use it without explicit reference in the rest of the paper.

\section{Commuting graph of a Rees matrix semigroup over a group} \label{Commuting graph of a Rees matrix semigroup over a group}

Let $G$ be a finite group, let $I$ and $\Lambda$ be finite index sets and let $P$ be a $\Lambda \times I$ matrix whose entries are elements of $G$. Since $G$, $I$ and $\Lambda$ are finite, then so is $\Rees{G}{I}{\Lambda}{P}$. 

In this section we characterize the commuting graph of $\Rees{G}{I}{\Lambda}{P}$ and study its properties. More specifically, we determine its diameter, clique number, girth and chromatic number, as well as the knit degree of $\Rees{G}{I}{\Lambda}{P}$. Studying the commuting graph of $\Rees{G}{I}{\Lambda}{P}$ allows us to find the maximum size of a commutative subsemigroup of $\Rees{G}{I}{\Lambda}{P}$, as well as the commutative subemigroups of $\Rees{G}{I}{\Lambda}{P}$ of that size.

When the index sets $I$ and $\Lambda$ are singletons, we have $\Rees{G}{I}{\Lambda}{P}\simeq G$ (see, for example, \cite[Exercise 4.1]{Nine_chapters_Cain}). Consequently, when $G$ is not abelian, the commuting graphs $\commgraph{\Rees{G}{I}{\Lambda}{P}}$ and $\commgraph{G}$ are isomorphic. Thus the questions of connectivity, diameter, clique number, girth, chromatic number and knit degree are equivalent to the corresponding questions for groups.

In the remainder of this section, we will see what the properties of $\commgraph{\Rees{G}{I}{\Lambda}{P}}$ look like when at least one of $I$ and $\Lambda$ is not a singleton. We begin this case by verifying that $\Rees{G}{I}{\Lambda}{P}$ is not commutative, which means we do not need to add any restrictions in order to ensure that its commuting graph is non-empty. In fact, because the center is empty, all the elements of $I \times G \times \Lambda$ are vertices of $\commgraph{\Rees{G}{I}{\Lambda}{P}}$.

\begin{proposition}\label{center M[G;I,^,P] non-commutative}
    Suppose that $\abs{I}>1$ or $\abs{\Lambda}>1$. Then $Z\parens{\Rees{G}{I}{\Lambda}{P}}\allowbreak=\emptyset$ and $\Rees{G}{I}{\Lambda}{I}$ is non-commutative.
\end{proposition}

\begin{proof}
    Let $i\in I$, $\lambda \in \Lambda$ and $x\in G$.
    
    Assume that $\abs{I}>1$. (The case where $\abs{\Lambda} >1$ is symmetrical.) Then there exists $j\in I$ such that $i\neq j$. Consequently, we have
    \begin{displaymath}
     \parens{i,x,\lambda}\parens{j,x,\lambda}=\parens{i,xp_{\lambda j}x,\lambda}\neq\parens{j,xp_{\lambda i}x,\lambda}=\parens{j,x,\lambda}\parens{i,x,\lambda}.   
    \end{displaymath}
    Thus $\parens{i,x,\lambda}\notin Z\parens{\Rees{G}{I}{\Lambda}{P}}$.

    Therefore no $\parens{i,x,\lambda}\in \Rees{G}{I}{\Lambda}{P}$ lies in the center of $\Rees{G}{I}{\Lambda}{P}$. In particular, $\Rees{G}{I}{\Lambda}{P}$ is not commutative.
\end{proof}

The following two lemmata describe how commutativity works in $\Rees{G}{I}{\Lambda}{P}$: more specifically, they provide necessary and sufficient conditions for two elements of $\Rees{G}{I}{\Lambda}{P}$ to commute. Lemma~\ref{commutativity} and Lemma~\ref{xy=yx <=> (,p^-1x,) (,p^-1y,) commute} will be used frequently in the remainder of this section and play an important role in the determination of the properties of the commuting graph of $\Rees{G}{I}{\Lambda}{P}$.

\begin{lemma}\label{commutativity}
    Let $i,j\in I$ and $\lambda,\mu\in\Lambda$ and $x,y\in G$. Then $\parens{i,x,\lambda}\parens{j,y,\mu}=\parens{j,y,\mu}\parens{i,x,\lambda}$ if and only if $i=j$, $\lambda=\mu$ and $xp_{\lambda i}y=yp_{\lambda i}x$.
\end{lemma}

\begin{proof}
    We have
    \begin{align*}
        & \parens{i,x,\lambda}\parens{j,y,\mu}=\parens{j,y,\mu}\parens{i,x,\lambda}\\
        \iff{} & \parens{i,xp_{\lambda j}y,\mu}=\parens{j,yp_{\mu i}x,\lambda}\\
        \iff{} & i=j \land xp_{\lambda j}y=yp_{\mu i}x \land \lambda=\mu\\
        \iff{} & i=j \land xp_{\lambda i}y=yp_{\lambda i}x \land \lambda=\mu. \qedhere
    \end{align*}
\end{proof}

\begin{lemma}\label{xy=yx <=> (,p^-1x,) (,p^-1y,) commute}
    Let $i\in I$ and $\lambda \in \Lambda$ and $x,y\in G$. Then $xy=yx$ if and only if $\parens{i,p_{\lambda i}^{-1}x,\lambda}\parens{i,p_{\lambda i}^{-1}y,\lambda}=\parens{i,p_{\lambda i}^{-1}y,\lambda}\parens{i,p_{\lambda i}^{-1}x,\lambda}$.
\end{lemma}

\begin{proof}
    We have
    \begin{align*}
        &xy=yx\\
        \iff{} & p_{\lambda i}^{-1}xy=p_{\lambda i}^{-1}yx\\
        \iff{} & \parens{p_{\lambda i}^{-1}x}p_{\lambda i}\parens{p_{\lambda i}^{-1}y}=\parens{p_{\lambda i}^{-1}y}p_{\lambda i}\parens{p_{\lambda i}^{-1}x} & \bracks{\text{by Lemma~\ref{commutativity}}}\\
        \iff{} & \parens{i,p_{\lambda i}^{-1}x,\lambda}\parens{i,p_{\lambda i}^{-1}y,\lambda}=\parens{i,p_{\lambda i}^{-1}y,\lambda}\parens{i,p_{\lambda i}^{-1}x,\lambda}.& \qedhere
    \end{align*}
\end{proof}

In Theorem~\ref{connectedness and diameter} we show that (when at least one of the sets $I$ and $\Lambda$ is not a singleton) $\commgraph{\Rees{G}{I}{\Lambda}{P}}$ is never connected. In addition, we identify and describe the connected components of $\commgraph{\Rees{G}{I}{\Lambda}{P}}$.

Furthermore, we observe that, it follows from Theorem~\ref{connectedness and diameter} that the matrix used to construct a Rees matrix semigroup over a group does not influence its commuting graph. That is, if $Q$ is also a $\Lambda\times I$ matrix whose entries are elements of $G$, then the graphs $\commgraph{\Rees{G}{I}{\Lambda}{P}}$ and $\commgraph{\Rees{G}{I}{\Lambda}{Q}}$ are isomorphic.

\begin{theorem}\label{connectedness and diameter}
    Suppose that $\abs{I}>1$ or $\abs{\Lambda}>1$. Then
    \begin{enumerate}
        \item $\commgraph{\Rees{G}{I}{\Lambda}{P}}$ is not connected and its connected components are the subgraphs induced by the $\mathcal{H}$-classes of $\Rees{G}{I}{\Lambda}{P}$.
        \item All the connected components of $\commgraph{\Rees{G}{I}{\Lambda}{P}}$ are isomorphic to $\extendedcommgraph{G}$ and their diameter is equal to
        \begin{displaymath}\begin{cases}
        0& \text{if } G \text{ is trivial,}\\
        1& \text{if } G \text{ is abelian and not trivial,}\\
        2& \text{if } G \text{ is not abelian}.
    \end{cases}
    \end{displaymath}
    \end{enumerate}
\end{theorem}

\begin{proof}
    The $\mathcal{H}$-classes of $\Rees{G}{I}{\Lambda}{P}$ are the sets $\braces{i} \times G \times \braces{\lambda}$, where $i\in I$ and $\lambda \in \Lambda$. For each $i \in I$ and $\lambda\in \Lambda$ we define $C_{i\lambda}$ as the subgraph of $\commgraph{\Rees{G}{I}{\Lambda}{P}}$ induced by the $\mathcal{H}$-class $\set{i}\times G \times \set{\lambda}$. We are going to see that those subgraphs are the connected components of $\commgraph{\Rees{G}{I}{\Lambda}{P}}$.

    Let $i\in I$ and $\lambda \in \Lambda$. First, we are going to prove that $C_{i\lambda}$ is a subgraph of a connected component of $\commgraph{\Rees{G}{I}{\Lambda}{P}}$. Let $x\in G\setminus\set{1_G}$. Since $1_G$ and $p_{\lambda i}x$ commute then, by Lemma~\ref{xy=yx <=> (,p^-1x,) (,p^-1y,) commute}, $\parens{i,p_{\lambda i}^{-1}\parens{p_{\lambda i}x},\lambda}$ and $\parens{i,p_{\lambda i}^{-1}1_G,\lambda}$ also commute. Hence $\parens{i,x,\lambda}=\parens{i,p_{\lambda i}^{-1}\parens{p_{\lambda i}x},\lambda},\parens{i,p_{\lambda i}^{-1}1_G,\lambda}=\parens{i,p_{\lambda i}^{-1},\lambda}$ is a path from $\parens{i,x,\lambda}$ to $\parens{i,p_{\lambda i}^{-1},\lambda}$. Thus all the vertices belonging to $\set{i}\times G\times\set{\lambda}$ are in the same connected component of $\commgraph{\Rees{G}{I}{\Lambda}{P}}$.
    
    Now we are going to prove that the connected component that contains the vertices belonging to $\set{i}\times G \times \set{\lambda}$ is the subgraph $C_{i\lambda}$. In order to do this, we are going to see that any path that begins with the vertex $\parens{i, p_{\lambda i}^{-1}, \lambda}$ ends in a vertex that belongs to $\set{i}\times G \times \set{\lambda}$. Let 
    \begin{displaymath}
        \parens{i, p_{\lambda i}^{-1}, \lambda}=\parens{i_1, x_1, \lambda_1},\parens{i_2, x_2, \lambda_2},\cdots,\parens{i_n, x_n, \lambda_n}
    \end{displaymath}
    be a path in $\commgraph{\Rees{G}{I}{\Lambda}{P}}$ whose initial vertex is $\parens{i, p_{\lambda i}^{-1}, \lambda}$. We have $\parens{i_j, x_j, \lambda_j}\parens{i_{j+1}, x_{j+1}, \lambda_{j+1}}=\parens{i_{j+1}, x_{j+1}, \lambda_{j+1}}\parens{i_j, x_j, \lambda_j}$ for all $j\in\set{1,\ldots,\allowbreak n-1}$. Therefore, by Lemma~\ref{commutativity}, $i_j=i_{j+1}$ and $\lambda_j=\lambda_{j+1}$ for all $j\in\set{1,\ldots,n-1}$, which implies that $i=i_1=i_2=\cdots=i_n$ and $\lambda=\lambda_1=\lambda_2=\cdots=\lambda_n$. Thus $\parens{i_n, x_n, \lambda_n}\in\set{i}\times G \times \set{\lambda}$.

    We just proved that the connected components of $\commgraph{\Rees{G}{I}{\Lambda}{P}}$ are the subgraphs $C_{i\lambda}$, where $i\in I$ and $\lambda \in \Lambda$. Note that since $\abs{I}>1$ or $\abs{\Lambda}>1$, then $\commgraph{\Rees{G}{I}{\Lambda}{P}}$ contains more than one connected component, which implies that $\commgraph{\Rees{G}{I}{\Lambda}{P}}$ is not connected. This concludes the proof of $(1)$.

     Let $i\in I$ and $\lambda \in \Lambda$. Our aim is to prove that $C_{i\lambda}$ and $\extendedcommgraph{G}$ are isomorphic. Let $\varphi: G \to \braces{i}\times G \times \braces{\lambda}$ be the map defined by $x\varphi=\parens{i,p_{\lambda i}^{-1}x, \lambda}$ for all $x\in G$. It is straightforward to see that $\varphi$ is a group isomorphism. Furthermore, $G$ and $\set{i}\times G\times\set{\lambda}$ are the sets of vertices of $\extendedcommgraph{G}$ and $C_{i\lambda}$, respectively. Thus, by Lemma~\ref{xy=yx <=> (,p^-1x,) (,p^-1y,) commute}, the graphs $\extendedcommgraph{G}$ and $C_{i\lambda}$ are isomorphic.

    We are going to determine the diameter of $C_{i\lambda}$. We begin by noticing that $\diam{C_{i\lambda}}=\diam{\extendedcommgraph{G}}$. If $G=\braces{1_G}$ then $\extendedcommgraph{G}$ has only one vertex and, therefore, $\diam{\extendedcommgraph{G}}=0$. If $G$ is abelian and not trivial, then $\extendedcommgraph{G}$ has more than one vertex and all its vertices are adjacent to each other, which implies that $\diam{\extendedcommgraph{G}}=1$. Now suppose that $G$ is not abelian. Since $1_G$ commutes with all the elements of $G$, then $1_G$ is adjacent to all the other vertices of $\extendedcommgraph{G}$. Thus $\diam{\extendedcommgraph{G}}\leqslant 2$. Moreover, there exist $x,y\in G$ such that $xy\neq yx$ (because $G$ is not abelian) and, consequenly, the vertices $x$ and $y$ are not adjacent in $\extendedcommgraph{G}$. Thus $\diam{\extendedcommgraph{G}}\geqslant 2$, which concludes the proof of $(2)$.
\end{proof}

We note that, as a result of part (1) of Theorem~\ref{connectedness and diameter} and the fact that the $\mathcal{H}$-classes of $\Rees{G}{I}{\Lambda}{P}$ are the sets $\braces{i}\times G\times \braces{\lambda}$, where $i\in I$ and $\lambda \in \Lambda$, it follows that the number of connected components of $\commgraph{\Rees{G}{I}{\Lambda}{P}}$ is $\abs{I}\cdot\abs{\Lambda}$ (when $\abs{I}>1$ or $\abs{\Lambda}>1$).

Moreover, as a consequence of Lemma~\ref{commgraph and extended commgraph} and Theorem~\ref{connectedness and diameter}, we know that if $G$ is abelian then all the connected components of $\commgraph{\Rees{G}{I}{\Lambda}{P}}$ are isomorphic to $K_{\abs{G}}$, and if $G$ is not abelian then all the connected components of $\commgraph{\Rees{G}{I}{\Lambda}{P}}$ are isomorphic to $K_{\abs{Z\parens{G}}}\graphjoin\commgraph{G}$.

 In the next result, we use the information we have on the structure of $\commgraph{\Rees{G}{I}{\Lambda}{P}}$ (given by Theorem~\ref{connectedness and diameter}) to establish a relation between the maximum size of a commutative subsemigroup of $\Rees{G}{I}{\Lambda}{P}$ and the maximum size of an abelian subgroup of $G$. Moreover, we identify the maximum-order commutative subsemigroups of $\Rees{G}{I}{\Lambda}{P}$. (We note that in Corollary~\ref{maximum size commutative semigroup M[G;I,^,P]} there are no restriction on the sizes of $I$ and $\Lambda$.)

\begin{corollary}\label{maximum size commutative semigroup M[G;I,^,P]}
    The maximum size of a commutative subsemigroup of $\Rees{G}{I}{\Lambda}{P}$ is equal to the maximum size of an abelian subgroup of $G$. Furthermore, the commutative subsemigroups of $\Rees{G}{I}{\Lambda}{P}$ of maximum size are the groups $\braces{i}\times \parens{p_{\lambda i}^{-1}H}\times\braces{\lambda}$, where $i\in I$, $\lambda\in\Lambda$ and $H$ is an abelian subgroup of $G$ of maximum size.
\end{corollary}

\begin{proof}
    For each $i\in I$ and $\lambda\in\Lambda$ let $\varphi_{i\lambda}:G\to \set{i}\times G \times\set{\lambda}$ be the map defined by $x\varphi_{i\lambda}=\parens{i,p_{\lambda i}^{-1}x, \lambda}$ for all $x\in G$. It is easy to see that for each $i\in I$ and $\lambda\in\Lambda$ the map $\varphi_{i\lambda}$ is an isomorphism. We consider the following two cases.

    Case 1: Suppose that $\abs{I}=\abs{\Lambda}=1$. It follows that $G\simeq\Rees{G}{I}{\Lambda}{P}$, which implies that $S$ is an abelian subgroup of $G$ if and only if $S\varphi_{i\lambda}=\set{i}\times \parens{p_{\lambda i}^{-1}S}\times\set{\lambda}$ is an abelian subgroup of $\Rees{G}{I}{\Lambda}{P}$, where $i\in I$ and $\lambda\in\Lambda$. Thus the maximum size of a commutative subsemigroup of $\Rees{G}{I}{\Lambda}{P}$ is equal to the maximum size of an abelian subgroup of $G$, and the largest commutative subsemigroups of $\Rees{G}{I}{\Lambda}{P}$ are the groups $\set{i}\times \parens{p_{\lambda i}^{-1}H}\times\set{\lambda}$, where $H$ is a maximum-order abelian subgroup of $G$.
    
    Case 2: Suppose that $\abs{I}>1$ or $\abs{\Lambda}>1$. For each $i\in I$ and $\lambda\in\Lambda$ let $C_{i\lambda}$ be the connected component of $\commgraph{\Rees{G}{I}{\Lambda}{P}}$ whose set of vertices is $\set{i}\times G \times\set{\lambda}$. Using the characterization of $\commgraph{\Rees{G}{I}{\Lambda}{P}}$ given by Theorem~\ref{connectedness and diameter}, and the fact that $Z\parens{\Rees{G}{I}{\Lambda}{P}}=\emptyset$, we have
    \begin{align*}
        & S \text{ is a maximal commutative subsemigroup of } \Rees{G}{I}{\Lambda}{P}\\
        \iff{} & S\setminus Z\parens{\Rees{G}{I}{\Lambda}{P}} \text{ is a maximal clique in } \commgraph{\Rees{G}{I}{\Lambda}{P}}\\
        \iff{} & S \text{ is a maximal clique in } \commgraph{\Rees{G}{I}{\Lambda}{P}}\\
        \iff{} & S \text{ is a maximal clique in } C_{i\lambda}, \text{ for some } i\in I \text{ and } \lambda\in\Lambda\\
        \iff{} & S\varphi_{i\lambda}^{-1} \text{ is a maximal clique in } \extendedcommgraph{G}, \text{ for some } i\in I \text{ and } \lambda\in\Lambda\\
        \iff{} & S\varphi_{i\lambda}^{-1} \text{ is a maximal abelian subgroup of } G, \text{ for some } i\in I \text{ and } \lambda\in\Lambda,
    \end{align*}
in which case $S=\parens{S\varphi_{i\lambda}^{-1}}\varphi_{i\lambda}=\set{i}\times \parens{p_{\lambda i}^{-1}\parens{S\varphi_{i\lambda}^{-1}}}\times\set{\lambda}$. This implies that the maximum size of a commutative subsemigroup of $\Rees{G}{I}{\Lambda}{P}$ is equal to the maximum size of an abelian subgroup of $G$, and the commutative subsemigroups of $\Rees{G}{I}{\Lambda}{P}$ of maximum size are the groups $\set{i}\times \parens{p_{\lambda i}^{-1}H}\times\set{\lambda}$, where $H$ is a maximum-order abelian subgroup of $G$, $i\in I$ and $\lambda\in\Lambda$.
\end{proof}

The rest of this section is devoted to obtaining the clique number (Theorem~\ref{clique number}), girth (Corollary~\ref{girth}) and chromatic number (Theorem~\ref{chromatic number}) of $\commgraph{\Rees{G}{I}{\Lambda}{P}}$ (when $\abs{I}>1$ or $\abs{\Lambda}>1$). Additionally, in Theorem~\ref{left paths and knit degree} we also study the existence of left paths in $\commgraph{\Rees{G}{I}{\Lambda}{P}}$.

\begin{theorem}\label{clique number}
    Suppose that $\abs{I}>1$ or $\abs{\Lambda}>1$. Then
    \begin{enumerate}
        \item If $G$ is abelian, then $\cliquenumber{\commgraph{\Rees{G}{I}{\Lambda}{P}}}=\abs{G}$.
        \item If $G$ is non-abelian, then $\cliquenumber{\commgraph{\Rees{G}{I}{\Lambda}{P}}}=\abs{Z\parens{G}}+\cliquenumber{\commgraph{G}}$.
    \end{enumerate}
\end{theorem}

\begin{proof}
    We know, by Theorem~\ref{connectedness and diameter}, that all the connected components of $\commgraph{\Rees{G}{I}{\Lambda}{P}}$ are isomorphic to $\extendedcommgraph{G}$. Furthermore, Lemma~\ref{commgraph and extended commgraph} also states that $\extendedcommgraph{G}$ is isomorphic to $K_{\abs{G}}$ when $G$ is abelian, and is isomorphic to $K_{\abs{Z\parens{G}}}\graphjoin\commgraph{G}$ when $G$ is not abelian. Additionally, we know that the clique number of a complete graph is equal to its number of vertices, and the clique number of the join of two graphs is equal to the sum of the clique numbers of those two graphs. Hence
    \begin{align*}
        &\cliquenumber{\commgraph{\Rees{G}{I}{\Lambda}{P}}}\\
        ={}&\max\left\{\cliquenumber{C}: C \text{ is a connected compenent of } \commgraph{\Rees{G}{I}{\Lambda}{P}}\right\} \kern -14.7mm \\
        ={}& \cliquenumber{\extendedcommgraph{G}} & \bracks{\text{by Theorem~\ref{connectedness and diameter}}}\\
        ={}&\begin{cases}
            \cliquenumber{K_{\abs{G}}}& \text{if } G \text{ is abelian,}\\
            \cliquenumber{K_{\abs{Z\parens{G}}}\graphjoin\commgraph{G}}& \text{if } G \text{ is not abelian}
        \end{cases}&  \bracks{\text{by Lemma~\ref{commgraph and extended commgraph}}}\\
        ={}&\begin{cases}
            \cliquenumber{K_{\abs{G}}}& \text{if } G \text{ is abelian,}\\
            \cliquenumber{K_{\abs{Z\parens{G}}}}+\cliquenumber{\commgraph{G}}& \text{if } G \text{ is not abelian}
        \end{cases}\\
        ={}&\begin{cases}
            \abs{G}& \text{if } G \text{ is abelian,}\\
            \abs{Z\parens{G}}+\cliquenumber{\commgraph{G}}& \text{if } G \text{ is not abelian}.
        \end{cases}& \qedhere
    \end{align*}
\end{proof}

In Theorem~\ref{girth extended commgraph G} we prove that the existence of cycles in $\extendedcommgraph{G}$ depends exclusively on the size of $G$. Additionally, we see that when $\extendedcommgraph{G}$ contains cycles there is only one possible value for its girth. In Corollary~\ref{girth} we see that, when at least one of the index sets $I$ and $\Lambda$ is not a singleton, the same happens with $\commgraph{\Rees{G}{I}{\Lambda}{P}}$.

\begin{theorem}\label{girth extended commgraph G}
    $\extendedcommgraph{G}$ contains cycles if and only if $\abs{G}\geqslant 3$, in which case $\girth{\extendedcommgraph{G}}=3$.
\end{theorem}

\begin{proof}

Suppose that $G$ is abelian. As a result of Lemma~\ref{commgraph and extended commgraph}, $\extendedcommgraph{G}$ is isomorphic to $K_{\abs{G}}$. If $\abs{G}\leqslant 2$, then $\extendedcommgraph{G}$ has at most two vertices, which means that $\extendedcommgraph{G}$ has no cycles. If $\abs{G}\geqslant 3$, then $\extendedcommgraph{G}$ contains at least three vertices adjacent to each other. Therefore $\extendedcommgraph{G}$ contains at least one cycle of length $3$ and, consequently, $\girth{\extendedcommgraph{G}}=3$.

Now suppose that $G$ is not abelian. (Then $\abs{G}\geqslant 3$.) There exist distinct $x,y\in G\setminus Z\parens{G}$ such that $xy\neq yx$. We consider the following two cases.

Case $1$: Suppose that $x\neq x^{-1}$ or $y\neq y^{-1}$. Assume, without loss of generality, that $x\neq x^{-1}$. Clearly, $x\neq 1_G$ and $x^{-1}\neq 1_G$. Since $x$, $x^{-1}$ and $1_G$ commute with each other then $x,x^{-1},1_G,x$ is a cycle (of length $3$) in $\extendedcommgraph{G}$.

Case $2$: Suppose that $x=x^{-1}$ and $y=y^{-1}$. It follows that $yx=y^{-1}x^{-1}=(xy)^{-1}$, which implies that $xy$ and $yx$ commute. Additionally, $xy$ and $yx$ also commute with $1_G$, and $xy$, $yx$ and $1_G$ are pairwise distinct. As a result, $xy, yx, 1_G, xy$ is a cycle (of length $3$) in $\extendedcommgraph{G}$.

    In both cases we proved that $\extendedcommgraph{G}$ contains a cycle of length $3$. Hence $\girth{\extendedcommgraph{G}}=3$.
\end{proof}

The next corollary is a direct consequence of Theorems~\ref{connectedness and diameter} and \ref{girth extended commgraph G}.

\begin{corollary}\label{girth}
    Suppose that $\abs{I}>1$ or $\abs{\Lambda}>1$. Then
    \begin{enumerate}
        \item If $\abs{G}\leqslant 2$, then $\commgraph{\Rees{G}{I}{\Lambda}{P}}$ has no cycles.
        \item If $\abs{G}\geqslant 3$, then $\commgraph{\Rees{G}{I}{\Lambda}{P}}$ contains a cycle of length $3$ and, consequently, $\girth{\commgraph{\Rees{G}{I}{\Lambda}{P}}}=3$.
    \end{enumerate}
\end{corollary}

\begin{theorem}\label{chromatic number}
    Suppose that $\abs{I}>1$ or $\abs{\Lambda}>1$. Then
    \begin{enumerate}
        \item If $G$ is abelian, then $\chromaticnumber{\commgraph{\Rees{G}{I}{\Lambda}{P}}}=\abs{G}$.
        \item If $G$ is non-abelian, then $\chromaticnumber{\commgraph{\Rees{G}{I}{\Lambda}{P}}}=\abs{Z\parens{G}}+\chromaticnumber{\commgraph{G}}$.
    \end{enumerate}
\end{theorem}

\begin{proof}
    We start by noticing that the chromatic number of a complete graph is equal its number of vertices, and the chromatic number of the join of two graphs is equal to the sum of the chromatic numbers of those two graphs.
    
    By Lemma~\ref{commgraph and extended commgraph} and Theorem~\ref{connectedness and diameter} we know that, if $G$ is abelian, then all the connected components of $\commgraph{\Rees{G}{I}{\Lambda}{P}}$ are isomorphic to $K_{\abs{G}}$ and, if $G$ is not abelian, then all the connected components of $\commgraph{\Rees{G}{I}{\Lambda}{P}}$ are isomorphic to $K_{\abs{Z\parens{G}}}\graphjoin\commgraph{G}$. Thus
    \begin{align*}
        &\chromaticnumber{\commgraph{\Rees{G}{I}{\Lambda}{P}}}\\
        ={}&\max\left\{\chromaticnumber{C}: C \text{ is a connected component of } \commgraph{\Rees{G}{I}{\Lambda}{P}}\right\}  \kern -14.8mm\\
        ={}& \chromaticnumber{\extendedcommgraph{G}}& \bracks{\text{by Theorem~\ref{connectedness and diameter}}}\\
        ={}&\begin{cases}
            \chromaticnumber{K_{\abs{G}}}& \text{if } G \text{ is abelian,}\\
            \chromaticnumber{K_{\abs{Z\parens{G}}}\graphjoin\commgraph{G}}& \text{if } G \text{ is not abelian}
        \end{cases}& \bracks{\text{by Lemma~\ref{commgraph and extended commgraph}}}\\
        ={}&\begin{cases}
            \chromaticnumber{K_{\abs{G}}}& \text{if } G \text{ is abelian,}\\
            \chromaticnumber{K_{\abs{Z\parens{G}}}}+\chromaticnumber{\commgraph{G}}& \text{if } G \text{ is not abelian}
        \end{cases}\\
        ={}&\begin{cases}
            \abs{G}& \text{if } G \text{ is abelian,}\\
            \abs{Z\parens{G}}+\chromaticnumber{\commgraph{G}}& \text{if } G \text{ is not abelian}.
        \end{cases}\qedhere
    \end{align*}
\end{proof}

\begin{theorem}\label{left paths and knit degree}
    Suppose that $\Rees{G}{I}{\Lambda}{P}$ is not commutative. Then $\commgraph{\Rees{G}{I}{\Lambda}{P}}$ contains no left paths.
\end{theorem}

\begin{proof}
    Assume, with the aim of obtaining a contradiction, that $\commgraph{\Rees{G}{I}{\Lambda}{P}}$ contains at least one left path. Let
    \begin{displaymath}
        \parens{i_1,x_1,\lambda_1},\parens{i_2,x_2,\lambda_2},\dots,\parens{i_k,x_k,\lambda_k}
    \end{displaymath}
be a left path in $\commgraph{\Rees{G}{I}{\Lambda}{P}}$.

    We have $\parens{i_j,x_j,\lambda_j}\parens{i_{j+1},x_{j+1},\lambda_{j+1}}=\parens{i_{j+1},x_{j+1},\lambda_{j+1}}\parens{i_j,x_j,\lambda_j}$ for all $j\in\set{1,\ldots,k-1}$. By Lemma~\ref{commutativity}, $i_j=i_{j+1}$ and $\lambda_j=\lambda_{j+1}$ for all $j\in\set{1,\ldots,k-1}$, which implies that $i_1=i_2=\cdots=i_k$ and $\lambda_1=\lambda_2=\cdots=\lambda_k$.

    We also have that $\parens{i_1,x_1,\lambda_1}\parens{i_1,x_1,\lambda_1}=\parens{i_k,x_k,\lambda_k}\parens{i_1,x_1,\lambda_1}$. Hence $\parens{i_1,x_1p_{\lambda_1 i_1}x_1,\lambda_1}=\parens{i_k,x_kp_{\lambda_k i_1}x_1,\lambda_1}$ and, as a consequence, $x_1p_{\lambda_1 i_1}x_1=x_kp_{\lambda_k i_1}x_1\allowbreak =x_kp_{\lambda_1 i_1}x_1$. Thus $x_1=x_k$. Therefore $\parens{i_1,x_1,\lambda_1}=\parens{i_k,x_k,\lambda_k}$, which is a contradiction. Hence $\commgraph{\Rees{G}{I}{\Lambda}{P}}$ contains no left paths.
\end{proof}

The previous result states that $\commgraph{\Rees{G}{I}{\Lambda}{P}}$ has no left paths. In particular, this is also true when $\abs{I}=\abs{\Lambda}=1$. In this case we have $\Rees{G}{I}{\Lambda}{P}\simeq G$ and, since $\commgraph{\Rees{G}{I}{\Lambda}{P}}$ has no left paths, then $\commgraph{G}$ has no left paths as well. This proves the following corollary.

\begin{corollary}
    The commuting graph of a finite non-abelian group does not contain left paths.
\end{corollary}

\section{Properties of commuting graphs of completely simple semigroups}\label{Properties of commuting graphs of completely simple semigroups}

This section is devoted to the study of the commuting graphs of completely simple semigroups. In particular, we characterize the graphs that arise as commuting graphs of completely simple semigroups. We also determine the possible values for the clique number, girth and chromatic number of the commuting graph of a completely simple semigroup, as well as the possible values for the knit degree of a completely simple semigroup.

It has been proved in \cite{Semigroup_whose_commuting_graph_has_diameter_n} that for each positive integer $n\geqslant 3$ there exists a group whose commuting graph has diameter equal to $n$. There are also groups such that the diameter of its commuting graph is equal to $2$: \cite{Commuting_graph_group_diameter_2} gave some examples of such groups. (Notice that there are no semigroups whose commuting graph has diameter equal to $1$.) Since groups are completely simple semigroups, then this means that for each positive integer $n\geqslant 2$ there is also a completely simple semigroup whose commuting graph has diameter equal to $n$. Using some of the results from Section~\ref{Commuting graph of a Rees matrix semigroup over a group}, as well as the characterization of completely simple groups with Rees matrix semigroups over groups, we can see that for each positive integer $n\geqslant 2$, the only completely simple semigroups whose commuting graphs have diameter $n$ are groups. 

Since the possible values for the diameter of the commuting graph of a completely simple semigroup are known, we focus on other properties: Corollary~\ref{clique number completely simple semigroup}, Theorem~\ref{girth completely simple semigroup}, Corollary~\ref{chromatic number completely simple semigroup} and Corollary~\ref{knit degree completely simple semigroup} concern the clique number, girth, chromatic number and knit degree, respectively.

\begin{corollary}\label{clique number completely simple semigroup}
    For each $n\in\mathbb{N}$, there is a completely simple semigroup whose commuting graph has clique number equal to $n$.
\end{corollary}

\begin{proof}
    Let $n\in\mathbb{N}$. We consider the cyclic group $C_n$ of order $n$, which is an abelian group. Let $I$ and $\Lambda$ be index sets with at least one being a non-singleton, and let $P$ be a $\Lambda\times I$ matrix with entries belonging to $C_n$. By Proposition~\ref{center M[G;I,^,P] non-commutative}, $\Rees{C_n}{I}{\Lambda}{P}$ is not commutative and by Theorem~\ref{clique number} $\cliquenumber{\commgraph{\Rees{C_n}{I}{\Lambda}{P}}}=\abs{C_n}=n$, which concludes the proof.
\end{proof}

Next we determine the possible values for the girth of the commuting graph of a group. We then use this result to show in Theorem~\ref{girth completely simple semigroup} that those are also the values that are possible for the commuting graph of a completely simple semigroup.

\begin{proposition}\label{girth commuting graph of a group}
    Suppose that $G$ is non-abelian. If $\commgraph{G}$ contains at least one cycle, then $\girth{\commgraph{G}}=3$.
\end{proposition}

\begin{proof}
    Let $x_1,x_2,\cdots,x_n,x_1$ be a cycle in $\commgraph{G}$ (where $n\geqslant 3)$. We consider the following two cases.

    Case 1: Assume that there exists $i\in\braces{1,\ldots,n}$ such that $x_i\neq x_i^{-1}$. Assume, without loss of generality, that $i=1$. We know that $x_1$ is adjacent to $x_2$ and $x_n$, which implies that $x_1$ commutes with $x_2$ and $x_n$. Therefore $x_1^{-1}$ also commutes with $x_2$ and $x_n$. Note that we have $x_1^{-1}\in G\setminus Z\parens{G}$ because $x_1\in G\setminus Z\parens{G}$ (and, consequently, $x_1^{-1}$ is a vertex of $\commgraph{G}$). Let $x\in\braces{x_2,x_n}\setminus\braces{x_1^{-1}}$. Then $x_1$, $x_1^{-1}$ and $x$ are pairwise distinct and they commute with each other. Hence $x_1,x_1^{-1},x,x_1$ is a cycle (of length $3$) in $\commgraph{G}$.

    Case 2: Assume that $x_i= x_i^{-1}$ for all $i\in\braces{1,\ldots,n}$. If $x_1x_3=x_3x_1$ then $x_1,x_2,x_3,x_1$ is a cycle (of length $3$) in $\commgraph{G}$.
    
    Now, assume that $x_1x_3\neq x_3x_1$. Then $\parens{x_1x_3}^{-1}=x_3^{-1}x_1^{-1}=x_3x_1\neq x_1x_3$. Since $x_2=x_2^{-1}$, then we also have $x_2\neq x_1x_3$ and $x_2\neq \parens{x_1x_3}^{-1}$. Notice that $\parens{x_1x_3}x_1=x_1\parens{x_3x_1}\neq x_1\parens{x_1x_3}$, which implies that $x_1x_3, \parens{x_1x_3}^{-1}\in G\setminus Z\parens{G}$ (that is, $x_1x_3$ and $\parens{x_1x_3}^{-1}$ are vertices of $\commgraph{G}$). Since $x_2$ is adjacent to $x_1$ and $x_3$, then $x_2$ commutes with $x_1$ and $x_3$. Thus $x_2$ commutes with $x_1x_3$, and consequently, $x_2$ also commutes with $\parens{x_1x_3}^{-1}$. Therefore $x_2,x_1x_3,\parens{x_1x_3}^{-1},x_2$ is a cycle (of length $3$) in $\commgraph{G}$.

    In both cases we exhibit a cycle of length $3$ in $\commgraph{G}$. Thus the girth of $\commgraph{G}$ is $3$.
\end{proof}

We observe that Proposition~\ref{girth commuting graph of a group} is a particular case of the next theorem.

\begin{theorem}\label{girth completely simple semigroup}
    Let $S$ be a non-commutative completely simple semigroup. If $\commgraph{S}$ contains at least one cycle, then $\girth{\commgraph{S}}=3$.
\end{theorem}

\begin{proof}
    There exist a group $G$, index sets $I$ and $\Lambda$, and a $\Lambda \times I$ matrix $P$ with entries from $G$ such that $S\simeq\Rees{G}{I}{\Lambda}{P}$. We consider two cases.

    First, we assume that $\abs{I}=\abs{\Lambda}=1$.  Then the graphs $\commgraph{\Rees{G}{I}{\Lambda}{P}}$ and $\commgraph{G}$ are isomorphic. Since $\commgraph{G}$ contains at least one cycle, it follows from Proposition~\ref{girth commuting graph of a group} that $\girth{\commgraph{G}}=3$, which implies that $\girth{\commgraph{S}}=3$.

    Now, assume that $\abs{I}>1$ or $\abs{\Lambda}>1$. Since $\commgraph{\Rees{G}{I}{\Lambda}{P}}$ contains cycles, then Corollary~\ref{girth} guarantees that $\girth{\commgraph{\Rees{G}{I}{\Lambda}{P}}}=3$. Thus $\girth{\commgraph{S}}=3$.
\end{proof}

We observe that it is also possible for a commuting graph of a completely simple semigroup not to have any cycles. In fact, Corollary~\ref{girth} shows that the commuting graphs of the completely simple semigroups $\Rees{G}{I}{\Lambda}{P}$ --- where $G$ is a group such that $\abs{G}\leqslant 2$, $I$ and $\Lambda$ are two index sets such that at least one of them is not a singleton, and $P$ is a $\Lambda \times I$ matrix whose entries are elements of $G$ --- have no cycles.

\begin{corollary}\label{chromatic number completely simple semigroup}
    For each $n\in\mathbb{N}$, there is a completely simple semigroup whose commuting graph has chromatic number equal to $n$.
\end{corollary}

\begin{proof}
    Let $n\in\mathbb{N}$. Consider the cyclic group $C_n$ of order $n$, which is an abelian group. Let $I$ and $\Lambda$ be two index sets with at least one being a non-singleton, and let $P$ be a $\Lambda \times I$ matrix such that all its entries belong to $C_n$. It follows from Proposition~\ref{center M[G;I,^,P] non-commutative} that $\Rees{C_n}{I}{\Lambda}{P}$ is not commutative. Additionally, from Theorem~\ref{chromatic number}  we have $\chromaticnumber{\commgraph{\Rees{C_n}{I}{\Lambda}{P}}}=\abs{C_n}=n$.
\end{proof}

The corollary below is an immediate consequence of Theorem~\ref{left paths and knit degree}.

\begin{corollary}\label{knit degree completely simple semigroup}
    Let $S$ be a non-commutative completely simple semigroup. Then $\commgraph{S}$ has no left paths.
\end{corollary}

Notice that, due to the fact that the commuting graph of a completely simple group never has left paths, there are no possible values for the knit degree of those semigroups.

The last result of this section is a characterization of the commuting graph of a completely simple semigroup. This description is made using the commuting graphs and the extended commuting graphs of groups.

\begin{theorem}\label{characterization commuting graph competely simple semigroup}
    A simple graph $\mathcal{G}$ is the commuting graph of a (finite non-commutative) completely simple semigroup if and only if one of the following conditions is verified:
    \begin{enumerate}
        \item $\mathcal{G}$ is the commuting graph of a group.
        \item $\mathcal{G}$ has more than one connected component and all of them are isomorphic to the extended commuting graph of a group.
    \end{enumerate}
\end{theorem}

\begin{proof}
    We begin by proving the forward implication. Suppose that $\mathcal{G}=\commgraph{S}$ for some completely simple semigroup $S$. Then there exist a group $G$, index sets $I$ and $\Lambda$, and a $\Lambda\times I$ matrix $P$ with entries belonging to $G$ such that $S\simeq \Rees{G}{I}{\Lambda}{P}$. We consider the following two cases.

    Case $1$: Assume that $\abs{I}=\abs{\Lambda}=1$. Then $\Rees{G}{I}{\Lambda}{P}\simeq G$ and, consequently, $S$ is a group.

    Case $2$: Assume that $\abs{I}>1$ or $\abs{\Lambda}>1$. As a result of Theorem~\ref{connectedness and diameter}, $\commgraph{\Rees{G}{I}{\Lambda}{P}}$ has $\abs{I}\cdot\abs{\Lambda}>1$ connected components, which are all isomorphic to $\extendedcommgraph{G}$. Hence $\mathcal{G}$ is isomorphic to a graph that has more than one connected component, and all of them are all isomorphic to $\extendedcommgraph{G}$.

    Now we are going to prove the reverse implication. If $\mathcal{G}=\commgraph{G}$ for some group $G$ then, since groups are completely simple semigroups, the result follows. Now assume that $\mathcal{G}$ has $m\geqslant 2$ connected components, which are all isomorphic to $\extendedcommgraph{G}$ for some group $G$. Let $I$ and $\Lambda$ be two sets such that $\abs{I}=1$ and $\abs{\Lambda}=m$, and let $P$ be a $\Lambda\times I$ matrix whose entries belong to $G$. Due to the fact that $\abs{\Lambda}=m>1$, then it follows from Theorem~\ref{connectedness and diameter} that $\commgraph{\Rees{G}{I}{\Lambda}{P}}$ has $\abs{I}\cdot\abs{\Lambda}=m$ connected components, all of which are isomorphic to $\extendedcommgraph{G}$.  Thus $\mathcal{G}$ and $\commgraph{\Rees{G}{I}{\Lambda}{P}}$ are isomorphic and $\mathcal{G}$ is isomorphic to the commuting graph of a completely simple group.
\end{proof}

We observe that, given the description of extended commuting graphs provided by Lemma~\ref{commgraph and extended commgraph}, Theorem~\ref{characterization commuting graph competely simple semigroup} can be rewritten in the following way: a simple graph $\mathcal{G}$ is the commuting graph of a (finite non-commutative) completely simple semigroup if and only if one of the following conditions is verified:
\begin{enumerate}
        \item $\mathcal{G}$ is the commuting graph of a group.
        \item $\mathcal{G}$ has more than one connected component and all of them are isomorphic to the complete graph $K_n$ for some $n\in\mathbb{N}$.
        \item $\mathcal{G}$ has more than one connected component and all of them are isomorphic to $K_{\abs{Z\parens{G}}}\graphjoin\commgraph{G}$ for some group $G$.
    \end{enumerate}

    Let $m\in\mathbb{N}$ be such that $m \geqslant 2$. In the proof of Theorem~\ref{characterization commuting graph competely simple semigroup} we constructed a completely simple semigroup whose commuting graph has $m$ connected components, all of which are isomorphic to $\extendedcommgraph{G}$, for some finite group $G$. We could have constructed a different completely simple semigroup whose commuting graph is isomorphic to the one previously described. Let $n, n'\in \mathbb{N}$ be such that $m=nn'$. Since $m\geqslant 2$, then $n\geqslant 2$ or $n'\geqslant 2$. Assume, without loss of generality, that $n\geqslant 2$. Let $I'$ and $\Lambda'$ be two sets such that $\abs{I'}=n$ and $\abs{\Lambda'}=n'$, and let $P'$ be a $\Lambda'\times I'$ matrix whose entries belong to $G$. By Theorem~\ref{connectedness and diameter}, $\commgraph{\Rees{G}{I'}{\Lambda'}{P'}}$ has $\abs{I'}\cdot\abs{\Lambda'}=nn'=m$ connected components, all of which are isomorphic to $\extendedcommgraph{G}$. Even though the completely simple semigroup constructed in the proof of Theorem~\ref{characterization commuting graph competely simple semigroup} and $\Rees{G}{I'}{\Lambda'}{P'}$ are not isomorphic, their commuting graphs are isomorphic.

    \bibliography{Commuting_graphs_of_completely_simple_semigroups} 
\bibliographystyle{alphaurl}

\end{document}